\documentclass[10pt, reqno]{amsart}
\usepackage{lipsum}
\usepackage{a4wide}
\usepackage{amssymb} 
\usepackage{amsmath}
\usepackage{amsthm} 
\usepackage{tikz} 
\usepackage{amscd}
\usepackage{hyperref}
\usepackage{enumitem}
\hypersetup{colorlinks,linkcolor={black},citecolor={black},urlcolor={black}}
\numberwithin{equation}{section} 

\newcommand{\id}{\operatorname{id}}

\newcommand{\ind}{\operatorname{ind}}

\newcommand{\res}{\operatorname{res}}

\newtheorem{theorem}{Theorem}[section]

\newtheorem{lemma}[theorem]{Lemma}
\newtheorem{remark}[theorem]{Remark}
\newtheorem{proposition}[theorem]{Proposition}

\title[Families over the integral Bernstein Center and Tate
cohomology]{Families over the integral Bernstein Center and Tate
  cohomology of local Base change lifts for ${\rm GL}_n(F)$}
\author{Sabyasachi Dhar and Santosh Nadimpalli} \date{\today}
\begin{document}
\maketitle
\begin{abstract}
  Let $p$ and $l$ be distinct odd primes, and let $F$ be a 
  $p$-adic field. Let $\pi$ be a generic
  smooth integral representation of ${\rm GL}_n(F)$ over an
  $\overline{\mathbb{Q}}_l$-vector space.  Let $E$ be a
  finite Galois extension of $F$ with $[E:F]=l$.  Let $\Pi$ be the
  base change lift of $\pi$ to the group ${\rm GL}_n(E)$. Let $\mathbb{W}^0(\Pi, \psi_E)$ be the
  lattice of $\overline{\mathbb{Z}}_l$-valued functions in the
  Whittaker model of $\Pi$, with respect to a standard
  ${\rm Gal}(E/F)$-equivaraint additive character
  $\psi_E:E\rightarrow \overline{\mathbb{Q}}_l^\times$. We show that
  the unique generic sub-quotient of the zero-th Tate cohomology group
  of $\mathbb{W}^0(\Pi, \psi_E)$ is isomorphic to the Frobenius twist
  of the unique generic sub-quotient of the mod-$l$ reduction of
  $\pi$. We first prove a version of this result for a family of
  smooth generic representations of ${\rm GL}_n(E)$ over the integral
  Bernstein center of ${\rm GL}_n(F)$. Our methods use the theory of
  Rankin-selberg convolutions and simple identities of local
  $\gamma$-factors. The results of this article remove the hypothesis
  that $l$ does not divide the pro-order of ${\rm GL}_{n-1}(F)$ in our
  previous work \cite{dhar2022tate}.
  \end{abstract}

\section{Introduction}
Let $G$ be a reductive group defined over a number field $K$.  Let
$\sigma$ be an order $l$ automorphism of $G$ defined over $K$, and let
$G^\sigma$ be the connected component of the fixed points of
$\sigma$.  In the seminal paper \cite{Venkatesh_Smith_theory},
D.Treumann and A. Venkatesh established the functoriality lifting of a
mod-$l$ automorphic form on $G^\sigma$ to a mod-$l$ automorphic form
on $G$; here mod-$l$ automorphic form is defined as a Hecke
eigenclass in the cohomology of congruence subgroup with
$k$-coefficients.  At the same time, they also made some conjectures for
representation theory of $p$-adic groups, and these conjectures
predict that local (Langlands) functoriality is compatible with Tate
cohomology for the action of $\langle \sigma\rangle$ (see
\cite[Section 6.3]{Venkatesh_Smith_theory}).  After initial
evidence in the context of local cyclic base change for depth-zero
cuspidal representations due to Ronchetti (\cite{Ronchetti_local_bc}),
Feng (\cite{Feng_2024}, \cite{feng2023modular}) made remarkable
progress towards these conjectures--using the recent advances due to
the works of V. Lafforgue and Fargues-Scholze. In this article, using
completely different line of arguments, we approach the conjectures in 
\cite[Section 6.3]{Venkatesh_Smith_theory} in the setting of the local Base change
lifting. We use local Rankin-Selberg zeta functions and their functional equations to study the Tate cohomology and linkage
principle. Our approach is in the
spirit of local converse theorems.  In our earlier work
\cite{dhar2022tate}, we assumed that the prime $l$ is large. In this
article, exploiting the theory of smooth representations in families,
we remove this hypothesis.

Let $(\Pi, W)$ be an irreducible smooth
$\overline{\mathbb{F}}_l$-representation of $G(F)$, where $G$ is a connected reductive group
over a $p$-adic field $F$.  Say $\sigma \in {\rm Aut}(G)$ is of order
$l$ and that $\Pi^\sigma$ is isomorphic to $\Pi$. The Tate cohomology
group $\widehat{H}^i(\sigma, V)$ is a representation of
$G^\sigma(F)$. Any irreducible sub-quotient $\pi$ of
$\widehat{H}^i(\sigma, W)$ is defined to be linked with $\Pi$. In the
article \cite{Venkatesh_Smith_theory}, the authors discovered that
linkage is compatible with local functoriality. In this
paper, we prove the Linkage conjectures in the local Base change
setting of ${\rm GL}_n$ for all odd $l$. The authors proved the main
theorem under the restriction that $l$ does not divide the pro-order
of ${\rm GL}_{n-1}(F)$. By using the theory of Co-Whittaker modules
and integral Bernstein center, developed by Emerton and Helm
(\cite{Emerton_Helm_LLC_families}, \cite{Curtis_Helm}) in an essential
way, we remove this hypothesis on $l$.

To state the main results of this article, some notations are in
order. Let $F$ be a finite extension of $\mathbb{Q}_p$ and let $E/F$
be a finite Galois extension of prime degree $l$. Let
$\overline{\mathcal{K}}$ be the algebraic closure of the fraction
field of $W(\overline{\mathbb{F}}_l)$ with ring of integers
$\mathcal{O}$. Let $(\pi, V)$ be an irreducible smooth representation
of ${\rm GL}_n(F)$, where $V$ is a vector space over
$\overline{\mathcal{K}}$.  We assume that $(\pi, V)$ is an integral
representation, i.e., $V$ admits a ${\rm GL}_n(F)$-invariant
$\mathcal{O}$-lattice.  Let $(\Pi, W)$ be the base change lifting of
$\pi$ to ${\rm GL}_n(E)$. The representation $\Pi$ is isomorphic to
$\Pi^\gamma$, where $\gamma$ is any generator of ${\rm Gal}(E/F)$. Let
us fix an isomorphism $T:\Pi\rightarrow \Pi^\gamma$ such that
$T^l=\id$. Note that $\Pi$ is an integral representation of
${\rm GL}_n(E)$.

If $\pi$ is generic, then $\Pi$ is also a generic representation. Let
us fix an additive character $\psi_F$ of $F$ and let $\psi_E$ be the
composition $\psi_F\circ{\rm Tr}$, where ${\rm Tr}$ is the trace map
of the extension $E/F$. Let $\mathbb{W}(\pi, \psi_F)$
(resp. $\mathbb{W}(\Pi, \psi_E)$) be the Whittaker model of $\pi$
(resp. $\Pi$). Vigneras showed that the $\mathcal{O}$-module
$\mathbb{W}^0(\Pi, \psi_E)$ is a ${\rm GL}_n(E)$-stable lattice and by
multiplicity one results, $\mathbb{W}^0(\Pi, \psi_E)$ is stable under
the Galois action of ${\rm Gal}(E/F)$ induced by its natural action on
${\rm GL}_n(E)$. Note that there exists a unique generic sub-quotient
of the mod-$l$ reduction of a generic integral representation $\pi$ of
${\rm GL}_n(K)$, where $K$ is a $p$-adic field and this is denoted by
$J_l(\pi)$.  It is easy to show that the Tate cohomology of
$\widehat{H}^0(\mathbb{W}^0(\Pi, \psi_E))$ has a unique generic
sub-quotient as a representation of ${\rm GL}_n(F)$ (\cite[Proposition
6.3]{dhar2022tate}). Let $V$ be a vector space over
$\overline{\mathbb{F}}_l$, and let $V^{(l)}$ be the vector space
$V\otimes_{{\rm Fr}}\overline{\mathbb{F}}_l$ where ${\rm Fr}$ be the
Frobenius automorphism of $\overline{\mathbb{F}}_l$. Ronchetti (\cite[Theorem 6]{Ronchetti_local_bc}) proved that the Tate cohomology group $\widehat{H}^1(\mathcal{L})$ is zero
for any ${\rm GL}_n(E)$ and ${\rm Gal}(E/F)$ stable lattice $\mathcal{L}$ in a cuspidal representation. When $l$ does not divide the pro-order of ${\rm GL}_n(F)$,
the authors showed the same vanishing result for generic representations (\cite[Corollary 8.2]{dhar2022tate}). We prove the
following theorem on the zeroth Tate cohomology:
\begin{theorem}\label{intro_main_thm}
  Let $l$ and $p$ be two distinct odd primes. Let $F$ be a p-adic
  field and let $E/F$ be a Galois extension of degree $l$.  Let $\pi$
  be an integral smooth generic representation of ${\rm GL}_n(F)$ over
  a $\overline{\mathbb{Q}}_l$-vector space.  Let $\Pi$ be a base
  change lifting of $\pi$ to ${\rm GL}_n(E)$. The unique generic
  sub-quotient of the Tate cohomology group
  $\widehat{H}^0(\mathbb{W}^0(\Pi, \psi_E))$ is isomorphic to
  $J_l(\pi)^{(l)}$.
\end{theorem}
In our previous work (\cite{dhar2022tate}), we proved the above theorem
when $l$ does not divide the pro-order of ${\rm GL}_{n-1}(F)$.  We use
the theory of Rankin--Selberg zeta integrals and their functional
equations. The hypothesis on the prime $l$ in \cite{dhar2022tate} is
needed for the completeness of mod-$l$ Whittaker models. In order to
remove this hypothesis on $l$, we use the theory of smooth
representations in families due to Emerton and Helm. This approach
deals with the nilpotents which arise in the mod-$l$ Zeta
integrals. So, we first prove a version of the above theorem over
families and use it to prove the above theorem. In order to define the
correct family, we needed the work of Helm on local converse theorem,
where he defines a map between the integral Bernstein center and a
certain deformation ring. This gave us the notation of Base change
lifting map between Bernstein centers.

We will begin with the description of our result on Tate cohomology of
a family of smooth representations of ${\rm GL}_n(E)$. Let
$\mathcal{M}_n(K)$ be the category of smooth
$W(\overline{\mathbb{F}}_l)[{\rm GL}_n(K)]$-modules, where $K$ is a
$p$-adic field, and the center of the category $\mathcal{M}_n(K)$ is
denoted by $\mathcal{Z}_n(K)$. The primitive idempotents of
$\mathcal{Z}_n(K)$ corresponds to an inertial class $[L, \sigma]$,
where $\sigma$ is a supercuspidal $\overline{\mathbb{F}}_l$
representation of the Levi-subgroup $L$ of ${\rm GL}_n(K)$
(\cite[Section 12]{Helm_Bernstein_center}). Let $e_\mathfrak{s}$ be a
primitive idempotent of $\mathcal{Z}_n(F)$. Using mod-$l$ semisimple
Langlands correspondence of Vigneras (\cite[Theorem
1.6]{Vigneras_semisimple_mod-l}), we get a smooth representation
$$\rho:W_F\rightarrow {\rm GL}_n(\overline{\mathbb{F}}_l)$$
associated with $(L, \sigma)$. Let $e_\mathfrak{r}$ be the primitive
idempotent in $\mathcal{Z}_n(E)$ associated with
$\res_{W_E}\rho$. Using Helm's construction of a homomorphism
between irreducible component of Bernstein centre to a deformation
ring, we obtain a map
$$z_{E/F}:e_{\mathfrak{r}}\mathcal{Z}_n(E)
\rightarrow e_\mathfrak{s}\mathcal{Z}_n(F)$$
which interpolates the cuspidal
support of Base change lifting (\cite[Chapter 1]{Arthur_Clozel_BC}).

To interpolate those generic representations of ${\rm GL}_n(E)$
which arise as a base change lift of representations with a given
cuspidal support, we  define
$$\mathcal{V}=(e_\mathfrak{r}\ind_{N_n(E)}^{{\rm GL}_n(E)}\psi_E^{-1})
\otimes_{z_{E/F}}e_\mathfrak{s}\mathcal{Z}_n(F).$$ The module
$\mathcal{V}$ is a co-Whittaker
$e_\mathfrak{s}\mathcal{Z}_n(F)[{\rm GL}_n(F)]$-module.  The theory of
Rankin-selberg $\gamma$-factors in families show that the
representation $\mathcal{V}$ is isomorphic to $\mathcal{V}^\gamma$,
for all $\gamma\in {\rm Gal}(E/F)$. Thus, the space of Whittaker
functions (containing functions on ${\rm GL}_n(E)/N_n(E)$ with values
in $e_\mathfrak{s}\mathcal{Z}_n$) is invariant under the action of
${\rm Gal}(E/F)$. We then prove that the Tate cohomology group
${\widehat{H}}^0(\mathbb{W}(\mathcal{V}, \psi_E))$ realises
$$[e_\mathfrak{s}(\ind_{N_n(F)}^{{\rm GL}_n(F)}\psi_F^{-1})
\otimes_{W(\overline{\mathbb{F}}_l)} \overline{\mathbb{F}}_l]^{(l)}$$
as a sub-quotient. The proof of this result is similar to our previous
result in \cite[Theorem 6.7]{dhar2022tate}, where we use
Rankin-Selberg local zeta integrals and their liftings from
homogeneous spaces over $F$ to the corresponding spaces over $E$.

Theorem \ref{intro_main_thm} is proved by a specialization
argument. The proofs of this theorem and the previous version in
families are inspired from the proofs of local converse theorems, as
in \cite[Proposition 7.5.2]{Aut_forms_GL3(I)} and \cite[Theorem
1.1]{Henniart_conv_thm}. Note that the above arguments in families and
our specialization arguments essentially use Kirillov models and
completeness of Whittaker models. The following is the essential
content. Let $X_F$ be the space ${\rm GL}_{n-1}(F)/N_{n-1}(F)$, and
let $H\in C_c(X_F, \overline{\mathbb{F}}_l)$. Typically, for some
fixed integer $k$ and a suitably chosen Whittaker function $W$ in the
Kirrilov model of $\Pi$, the function $H$ is equal to
$$[\widehat{H}^0(\pi_E)(w)-\pi_F^{(l)}(w)]W\left(
\begin{pmatrix}g&0\\0&1\end{pmatrix}\right)$$
when ${\rm det}(g)=k$ and zero otherwise. Essentially we need
to show that $H$ is zero. If not,
there exists a primitive idempotent 
$e_\mathfrak{s}\in \mathcal{Z}_n(F)$ and a
$e_\mathfrak{s}\mathcal{Z}_n(F)$-valued Whittaker function 
$$W'\in \mathbb{W}(e_\mathfrak{s}
\ind_{N_{n-1}(F)}^{{\rm GL}_{n-1}(F)}\psi_F^{-1} , \psi_F^{-1})$$ such
that $$\int_{X_F}H(g)\otimes W'(g)\,d\mu\neq 0.$$ When the ring
$e_\mathfrak{s}\mathcal{Z}_n(F)\otimes \overline{\mathbb{F}}_l$ is
reduced, we can simply use Rankin-Selberg theory of mod-$l$
representations to prove the vanishing of $H$ and deduce our main
theorem. This has been the main point of our work \cite[Theorem
6.7]{dhar2022tate}. In this article, to overcome the problems with the
reducibility of the ring
$e_\mathfrak{s}\mathcal{Z}_n(F)\otimes \overline{\mathbb{F}}_l$, we
use our result on families.

\section{Integral Bernstein Center and Base change}
In this section, we recall the theory of integral Bernstein centers
and define a base change map between Bernstein centers--which is
compatible with the local base change lifting of irreducible smooth
representations of ${\rm GL}_n$. We essentially follow the works of
Emerton--Helm (\cite{Emerton_Helm_LLC_families}) and Helm
(\cite{Curtis_Helm}).
\subsection{}
For any finite extension $K$ of $\mathbb{Q}_p$, we denote by
$\mathfrak{o}_K$ its ring of integers and the maximal ideal of
$\mathfrak{o}_K$ is denoted by $\mathfrak{p}_K$. Let $q_K$ be the
cardinality of the residue field $\mathfrak{o}_K/\mathfrak{p}_K$. The
Weil group of $K$ is denoted by $W_K$ and its inertia subgroup is
denoted by $I_K$. The prime to $l$-subgroup of $I_K$ is denoted by
$I_K^{(l)}$. We denote by $G_n(K)$ the group ${\rm GL}_n(K)$ and it is
equipped with the natural topology induced from that of $K$.  Let $F$
be a finite extension of $\mathbb{Q}_p$ and let $E$ be a finite Galois
extension of prime degree $l$, where $p$ and $l$ are distinct odd
primes. The Galois group ${\rm Gal}(E/F)$ is denoted by $\Gamma$. Let
$\mathbb{F}$ be the algebraic closure of the finite field of $l$
elements and let $W(\mathbb{F})$ be the ring of Witt vectors; we use
the notation $\Lambda$ for $W(\mathbb{F})$. Let
$\overline{\mathcal{K}}$ be the algebraic closure of the field of
fraction of $\Lambda$ with ring of integers $\mathcal{O}$. Let
$\psi_F : F\rightarrow \Lambda^\times$ be a non-trivial additive
character and let $\psi_E$ be the character
$\psi_F\circ {\rm Tr}_{E/F}$, where ${\rm Tr}_{E/F}$ denotes the trace
map of the extension $E/F$.
\subsection{}
Let $R$ be a commutative ring with unity and let $G$ be the
$K$-rational points of a reductive algebraic group defined over
$K$. We denote by $\mathcal{M}_R(G)$ the category of smooth
$R[G]$-modules (the stabilizer of every element under the action of
$G$ is open), and the Bernstein center of $\mathcal{M}_R(G)$ is
denoted by $\mathcal{Z}_R(G)$. We assume that $G$ is the group
$G_n(K)$. We use the notation $\mathcal{Z}_n^K$ for the ring
$\mathcal{Z}_{\Lambda}(G_n(K))$. Let $P$ be a parabolic subgroup of
$G$. Let $M$ be a Levi-subgroup of $P$ and let $U$ be the unipotent
radical of $P$. Let $(\sigma, W)$ be a supercuspidal representation of
$M$, where $W$ is an $\mathbb{F}$-vector space. Let
$\mathfrak{s}=[M, \sigma]$ be the inertial equivalence class
containing the pair $(M, \sigma)$. Let $\mathcal{M}_\mathfrak{s}(G)$
be the full subcategory of $\mathcal{M}_\Lambda(G)$ consisting of all
smooth $\Lambda[G]$ modules such that any irreducible subquotient has
mod-$l$ (see \cite[Definition 4.12]{Helm_Bernstein_center}) inertial
supercuspidal support equal to $\mathfrak{s}$. Given any primitive
idempotent $e$ of $\mathcal{Z}_n^K$, there exists a unique inertial
equivalence class $\mathfrak{s}$ such that the Bernstein center of
$\mathcal{M}_\mathfrak{s}(G)$ is equal to $e\mathcal{Z}_n^K$. The
primitive idempotents of $\mathcal{Z}_n^K$ correspond to the inertial
classes of the form $[M, \sigma]$, where $\sigma$ is a supercuspidal
representation of $M$ over an $\mathbb{F}$-vector space.

\subsection{}
Helm in his foundational work \cite{Curtis_Helm} described integral
Bernstein center in terms of local Galois deformation rings. He
constructed a map between connected component of the integral
Bernstein center and a local Galois deformation ring which
interpolates the local Langlands correspondence. Helm's constructions
naturally defines a map between Bernstein centers $\mathcal{Z}_n^E$
and $\mathcal{Z}_n^F$, which interpolates the compatibility of
cuspidal support maps of base change lifting. For the cyclic extension
$E/F$, the base change lifting of smooth irreducible representations
of $G$ over complex vector spaces is characterized via certain
character identities (see \cite[Chapter 3]{Arthur_Clozel_BC}). There
is a relation between the local base change lifting of irreducible
smooth $\overline{\mathcal{K}}$-representations of $G$ and the local
Langlands correspondence over $\overline{\mathcal{K}}$. Here, we use
the local Langlands correspondence by fixing the normalizations as in
\cite[Section 4.2]{Emerton_Helm_LLC_families}.
\subsubsection{}
An $l$-inertial type is a representation
$\nu:I_K^{(l)}\rightarrow {\rm GL}_n(\Lambda)$ that extends to a
representation of $\mathcal{W}_F$. Let $\mathfrak{s}$ be the inertial
equivalence class $[M, \sigma]$.  where $M$ is a Levi subgroup of $G$
and $\sigma$ is a supercuspidal representation of $M$ over an
$\mathbb{F}$-vector space. Vigneras' construction (\cite[Theorem
1.6]{Vigneras_semisimple_mod-l}) of mod-$l$ local Langlands
correspondence attaches a unique semisimple representation
$\overline{\rho}:W_K\rightarrow {\rm GL}_n(\mathbb{F})$ with the pair
$(M, \sigma)$
$$
(M, \sigma)\leftrightsquigarrow\overline{\rho}\ (\text{mod}-l\ \text{LLC})
$$
Let $\overline{\tau}$ be an irreducible representation of $I_K^{(l)}$
over an $\mathbb{F}$-vector space, and let $\tau$ be the (unique) lift
of $\overline{\tau}$ over $\Lambda$. Given any $\Lambda$-algebra $A$
and a representation $\rho:W_K\rightarrow {\rm GL}_n(A)$, we have
$$
\rho_A\simeq \bigoplus_{[\overline{\tau}]}({\rm Hom}_A(\tau,
\rho_A)\otimes \tau).
$$
The direct sum is over the $W_K$-conjugacy classes of
$I_K^{(l)}$-representations, denoted as $[\overline{\tau}]$.  The
space ${\rm Hom}_A(\tau, \rho)$ is a free $A$-module and a
pseudoframing of $\rho_A$ is a choice of a basis for each
${\rm Hom}_A(\tau, \rho_A)$. In the article \cite[Section
8]{Curtis_Helm}, the author constructed a $\Lambda$-algebra $R_\nu$
corresponding to an $l$-inertial type $\nu$--which is universal for
the pseudoframed deformation $\rho_A:W_K\rightarrow {\rm GL}_n(A)$,
i.e., for any choice of basis of ${\rm Hom}_{I_K^{(l)}}(\tau, \rho_A)$
which lifts a basis of
${\rm Hom}_{I_K^{(l)}}(\overline{\tau}, \overline{\rho})$ for each
$W_F$-conjugacy classes of $I_K^{(l)}$-representations, denoted by
$[\overline{\tau}]$, there exists a map $R_\nu^K\rightarrow A$ such
that the pseudoframed deformation $\rho_A$ is obtained as a base
change of the universal pseudoframed deformation $\rho_\nu^K$. The
affine scheme corresponding to the deformation ring $R_\nu^K$ is
equipped with an action of a group $G_\nu^K$, defined in \cite[Section
8]{Curtis_Helm}, which acts by change of framing; we recall the
precise definition of $G_\nu^K$ in the proof of Lemma
\ref{Galois_base_change_map}. This action gives a space of invariance
$(R_\nu^K)^{\rm inv}$, which is infact a subalgebra of $R_\nu^K$. The
space $(R_\nu^K)^{\rm inv}$ is useful to construct the base change map
between the integral Bernstein centers, as we will see in the next
subsections.
\subsubsection{}
In \cite[Section 7]{Conv_thm_Helm_Moss}, the authors constructed a map
$$
\mathbb{L}^K_\nu :
e\mathcal{Z}_n^K\rightarrow R^K_\nu
$$ 
which is compatible with the local Langlands correspondance, i.e., for
any morphism $$x:R_\nu^K\rightarrow\overline{\mathcal{K}},$$ the
morphism
$$x\circ \mathbb{L}_\nu^K:e\mathcal{Z}_n^K
\rightarrow \overline{\mathcal{K}}$$ gives the action of
$e\mathcal{Z}_n^K$ on the representation $\pi_x$--associated with
$\rho_x$ via the local Langlands correspondence. The map
$\mathbb{L}_\nu^K$ identifies the algebra $e\mathcal{Z}_n^K$ with the
$\mathcal{G}_\nu$-fixed points of $R_\nu^K$. This morphism plays a
fundamental role in proving the local Langlands in families by Emerton
and Helm (\cite[Conjecture 1.3.1]{Emerton_Helm_LLC_families}). Let $F$
be a finite extension of $\mathbb{Q}_p$ and let $E$ be a finite
extension of $F$ of degree $l$ with $l\neq p$. Let
$\nu:I_F^{(l)}\rightarrow {\rm GL}_n(\Lambda)$ be an $l$-inertial
type. As the group $I_F^{(l)}$ is equal to $I_E^{(l)}$, any
pseudo-framed deformation $\rho_A:W_F\rightarrow {\rm GL}_n(\Lambda)$
also determines a pseudoframed deformation
$\rho_A:W_E\rightarrow {\rm GL}_n(\Lambda)$ by restriction and hence
we get a map
\begin{equation}
B_{E/F}:R_\nu^E\rightarrow R_\nu^F. 
\end{equation}
\begin{lemma}\label{Galois_base_change_map}
The map $B_{E/F}$ induces a map between $(R_\nu^E)^{\rm inv}$ and
$(R_\nu^F)^{\rm inv}$.
\end{lemma}
\begin{proof}
  Let $X_{q,m}$ be the affine $\Lambda$-scheme parametrizing the pairs
  $(g,h)$, where $g$, $h$ are the invertible $m$ by $m$ matrices with
  $ghg^{-1} = h^q$. Let $R_{q,m}$ be the ring of functions of the
  connected compotent of $X_{q,m}$ containing the pair
  $({\rm Id}_n,{\rm Id}_n)$.  Note that the $\Lambda$-algebra
  $R_\nu^F$ is isomorphic to
  $\bigotimes_{[\overline{\tau}]}R_{q_{\overline{\tau}},
    n_{\overline{\tau}}}$, where $n_{\overline{\tau}}$ is the
  dimension of the space
  ${\rm Hom}_{I_F^{(l)}}(\overline{\tau}, \overline{\nu})$. The group
  ${\rm GL}_{n_{\overline{\tau}}}$ acts on
  $R_{q_{\overline{\tau}}, n_{\overline{\tau}}}$ by change of
  frame. Let $S_E=\{\tau_1,\tau_2,\dots, \tau_k\}$ (resp.
  $S_F\subset S_E$) be a set of representatives for the equivalence
  classes of $I_F^{(l)}$ representations for the action of $W_E$
  (resp. $W_F$). The group
$$
\prod_{\overline{\tau}\in S_E}{\rm GL}_{n_{\overline{\tau}}}
$$
denoted by $\mathcal{G}_{\nu}^E$, acts on $R^E_\nu$ via its action on
$R_{q_{\overline{\tau}}, n_{\overline{\tau}}}$. Similarly the group
$\mathcal{G}_\nu^F$, isomorphic to
$\prod_{\overline{\tau}\in S_F}{\rm GL}_{n_{\overline{\tau}}}$ is
contained in $\mathcal{G}_\nu^E$. Thus, the image of the restriction
of $B_{E/F}$ to $(R_\nu^E)^{\rm inv}$ is contained in the space of
invariance $(R_\nu^F)^{\rm inv}$.
\end{proof}

\subsubsection{}\label{base_change_map}
Let $e_{\mathfrak{s}}$ be a primitive idempotent of $\mathcal{Z}_n^F$
corresponding to the inertial equivalence class $\mathfrak{s}$. Fix a
pair $(L,\sigma)$ in the class $\mathfrak{s}$, where $L$ is a Levi
subgroup of $G_n(F)$ and $\sigma$ is an irreducible supercuspidal
$\mathbb{F}$-representation of $L$. Let $\rho_F$ be the
$n$-dimensional semisimple representation of $W_F$ associated with
$\sigma$ via the mod-$l$ semisimple local Langlands correspondence
(\cite[Theorem 1.6]{Vigneras_semisimple_mod-l}). Let
$\nu:I_F^{(l)}\rightarrow {\rm GL}_n(\Lambda)$ be the $l$-inertial
type such that its mod-$l$ reduction is isomorphic to
${\rm res}_{I_F^{(l)}}(\rho_F)$. Note that the restriction
${\rm res}_{W_E}(\rho_F)$ via mod-$l$ local Langlands correspondence
defines a pair $(L',\sigma')$ such that its inertial equivalence
class, denoted by $\mathfrak{r}$, is independent of the choice of
$(L,\sigma)$. Let $e_{\mathfrak{r}}$ be the primitive idempotent of
$\mathcal{Z}_n^E$ associated with $\mathfrak{r}$. The map $B_{E/F}$
induces a map
$z_{E/F}:e_\mathfrak{r}\mathcal{Z}_n^E\rightarrow
e_\mathfrak{s}\mathcal{Z}_n^F$ such that the following diagram
commutes.
$$\begin{tikzpicture}[scale=2]
  \node(A) at (0,1){$e_\mathfrak{r}\mathcal{Z}_n^E$};
  \node(B) at (1, 1){$R^E_\nu$};
  \node(C) at (0, 0){$e_\mathfrak{s}\mathcal{Z}_n^F$};
  \node(D) at (1, 0){$R^F_\nu$};
  \path[->, font=\scriptsize]
  (A) edge node[above]{$\mathbb{L}_\nu^E$} (B)
(A) edge node[left]{$z_{E/F}$} (C)
(B) edge node[right]{$B_{E/F}$} (D)
(C) edge node[above]{$\mathbb{L}_\nu^F$} (D);
\end{tikzpicture}
$$
Let $\rho_{x_F}:W_F\rightarrow {\rm GL}_n(\overline{\mathcal{K}})$ be the
representation corresponding to
$x_F: R_\nu^F\rightarrow \overline{\mathcal{K}}$ and let $\pi_{x_F}$
be a smooth irreducible representation of ${\rm GL}_n(F)$ associated
to $\rho_{x_F}$ via, the local Langlands correspondence. Let $x_E$ be
the map $x_F\circ B_{E/F}: R_\nu^E\rightarrow \overline{\mathcal{K}}$,
and the representation $\pi_{x_E}$ is the base change lifting of
$\pi_{x_F}$. The action of $\mathcal{Z}_n^E$ on the representation
$\pi_{x_E}$ factorises through $e_\mathfrak{r}\mathcal{Z}_n^E$, and it
is given by the homomorphism
$x_{F}\circ \mathbb{L}_{\nu}^F\circ z_{E/F}$.

\section{Rankin--Selberg convolution in families}
In this section, we review the theory of Rankin-Selberg convolution over arbitrary
Noetherian $\Lambda$-algebras. We start by recalling the notion of
co-Whittaker modules. For a precise reference, see
\cite{Conv_thm_Helm_Moss}, \cite{Gamma_factors_pairs_Moss}.
\subsection{Co-Whittaker modules}
Fix a non-trivial additive character
$\psi_K:K\rightarrow \Lambda^\times$. For any Noetherian
$\Lambda$-algebra $R$, we denote by $\psi_{K,R}$ the composition
$$
K\xrightarrow{\psi_K} \Lambda^\times\rightarrow R^\times.
$$
Let $N_n(K)$ be the group of all unipotent upper triangular matrices
in $G_n(K)$. The character $\psi_{K,R}:K\rightarrow R^\times$ induces
a character of $N_n(K)$, defined as
$$
(x_{ij})_{i,j=1}^n \longmapsto \psi_{K,R}(x_{12}+x_{23}+\cdots x_{n-1,n}),
$$
which is denoted by $\Psi_{K,R}$. Let $(\pi,V)$ be a smooth
$R[G_n(K)]$-module. We denote by $V^{(n)}$ the space of
$\Psi_{K,R}$-coinvariants, which is defined as the quotient of $V$
modulo the $R$-submodule generated by the set
$\Big\{\pi(x)v-(\Psi_{K,R})(x)v:x\in N_n(K), v\in V\Big\}$. The
$R$-module $V^{(n)}$ is the $n$-th Bernstein-Zelevinski derivative of
$V$.

\subsubsection{}
A smooth admissible $R[G_n(K)]$-module $(\pi,V)$ is called {\it
  co-Whittaker} if 
  \begin{enumerate}
      \item the $n$-th derivative $V^{(n)}$ is a free $R$-module
of rank one,
\item for any quotient $W$ of $V$ with $W^{(n)}=0$, we have $W=0$.
  \end{enumerate}
  A co-Whittaker module $(\pi, V)$ admits a central character denoted
  by $\varpi_\pi:F^\times\rightarrow R^\times$. By definition, the
  representation $(\pi,V)$ admits an $R$-module isomorphism
  $V^{(n)}\simeq R$. By Frobenius reciprocity, this isomorphism
  induces a $G_n(K)$-equivariant homomorphism
$$ V \longrightarrow {\rm Ind}_{N_n(K)}^{G_n(K)}(\Psi_{K,R}). $$
The image of $V$ under the above map is called the {\it Whittaker}
space of $\pi$. It is denoted by $\mathbb{W}(\pi,\psi_{K,R})$ and is
independent of the choice of isomorphism $V^{(n)} \simeq R$. The
identification $V^{(n)}\simeq R$ , also induces a $P_n(K)$-equivariant
map
$$ V \longrightarrow {\rm Ind}_{N_n(K)}^{P_n(K)}(\Psi_{K,R}),
$$ 
the image of the above map is called {\it Kirillov} space of $\pi$,
and it is denoted by $\mathbb{K}(\pi,\psi_{K,R})$. The map
$W\mapsto {\rm res}_{P_n(K)}(W)$ gives an isomorphism
$\mathbb{W}(\pi,\psi_{K,R})\xrightarrow{\sim}\mathbb{K}(\pi,\psi_{K,R})$
(\cite[Section 4]{Moss_Matringe_Kirillov_families}). Moreover, the
space $\mathbb{K}(\pi,\psi_{K,R})$ contains
${\rm ind}_{N_n(K)}^{P_n(K)}(\Psi_{K,R})$ as $R[P_n(K)]$-submodule.
\subsubsection{}
For a co-Whittaker $R[G_n(K)]$-module $(\pi,V)$, the endomorphism ring
${\rm End}_{R[G_n(K)]}(V)$ is equal to $R$, and the action of the
integral Bernstein center $\mathcal{Z}_n^K$ on $V$ induces a map
$f_\pi:\mathcal{Z}_n^K\rightarrow R$. The map $f_\pi$ is called the
{\it supercuspidal support} of $\pi$. The supercuspidal supports of
two co-Whittaker $R[G_n(K)]$-modules are the same if and only if they
have the same Whittaker spaces (see \cite[Lemma
2.4]{Converse_thm_Moss_Liu}). Moreover, if we have two co-Whittaker
$R[G_n(K)]$-modules $(\pi_1,V_1)$ and $(\pi_2,V_2)$ with a
$G_n(K)$-equivariant surjection $V_1\rightarrow V_2$, then
$f_{\pi_1} = f_{\pi_2}$.
\subsubsection{}
Let $W_n^K$ be the smooth $\Lambda[G_n(K)]$-module
${\rm ind}_{N_n(K)}^{G_n(K)}(\Psi_K^{-1})$. For a primitive idempotent
$e$ in the integral Bernstein center $\mathcal{Z}_n^K$, the space
$eW_n^K$ is a co-Whittaker $e\mathcal{Z}_n^K[G_n(K)]$-module.  The
following result is due to \cite[Theorem 6.3]{Whittaker_model_Helm},
which specify the universal property of $eW_n^K$.
\begin{theorem}
  Let $R$ be a Noetherian $\Lambda$-algebra, equipped with a
  $e\mathcal{Z}_n^K$-algebra structure. Then
  $eW_n^K\otimes_{e\mathcal{Z}_n^K} R$ is a co-Whittaker
  $R[G_n(K)]$-module. Conversely, if $(\pi,V)$ is a co-Whittaker
  $R[G_n(K)]$-module in the category $\mathcal{M}_e(G_n(K))$, then
  there exists a surjection
  $\beta:eW_n^K\otimes_{e\mathcal{Z}_n^K} R\rightarrow V$ such that
  the induced map
  $\beta^{(n)}:(eW_n^K\otimes_{e\mathcal{Z}_n^K} R)^{(n)}\rightarrow
  V^{(n)}$ is an isomorphism.
\end{theorem}
We end this subsection with a version of key vanishing result, usually
known as completeness of Whittaker models, provided by G.Moss
(\cite[Corollary 4.3]{Moss_nilpotent_gammafactors}). This vanishing
result will be crucially used in proving the main results of this
article.
\begin{theorem}\label{van_result}
  Let $R$ be a $\Lambda$-algebra, and let $\phi$ be an element of
  ${\rm ind}_{N_n(K)}^{G_n(K)}(\Psi_{K,R})$. If
$$
\int_{N_n(K)\setminus G_n(K)} \phi(g)\otimes W(g)\,dg = 0,
$$
for all $W\in\mathbb{W}(eW_n^K,\psi_K^{-1})$ and for all primitive
idempotents $e$ of $\mathcal{Z}_n^K$, then $\phi = 0$.
\end{theorem}

\subsection{Rankin-Selberg formal series}
In this subsection, we recall the Rankin-Selberg gamma factors over
families, which provide the classical Rankin-Selberg gamma factors as
a specialization on $\overline{\mathcal{K}}$-points. For a reference,
see \cite[Section 3]{Gamma_factors_pairs_Moss}.

\subsubsection{}
We now introduce some notations. Let $w_n$ be the
matrix of $G_n(K)$ of the form:
$$ 
w_n =
\begin{pmatrix}
        0  &  & &  1  \\
        &  & . \\
        & .   \\
        1 &  &  & 0
\end{pmatrix}
$$
For $r\in\mathbb{Z}$, set
$G_{n-1}(K)_r=\big\{g\in G_{n-1}(K) : v_K({\rm det}(g))=r\big\}$,
where $v_K$ denote the normalised discrete valuation of $K$. Let $X_K$
be the coset space $N_{n-1}(K)\setminus G_{n-1}(K)$. For an integer
$r$, we denote by $X_r^K$ the set of the form
$\big\{N_{n-1}(K)g:g\in G_{n-1}(K)_r\big\}$.

\subsubsection{}
Let $A$ and $B$ be two Noetherian $\Lambda$-algebras. Let $\pi$
(resp. $\pi'$) be the co-Whittaker $A[G_n(K)]$
(resp. $B[G_{n-1}(K)]$)-module. For any
$W\in\mathbb{W}(\pi,\psi_{K,A})$ and
$W'\in\mathbb{W}(\pi',\psi_{K,B}^{-1})$, the integral
$$
c_r^K(W,W'):=\int_{X_r^K}
W\begin{pmatrix}
    g & 0 \\
    0 & 1
\end{pmatrix}\otimes W'(g)\,dg
$$
is well-defined for all integers $r$ and it is zero for $r<<0$ (see
\cite[Section 3]{Gamma_factors_pairs_Moss}). The formal Laurent series
$$ 
\sum_{r\in\mathbb{Z}}c_r^K(W,W')X^r, 
$$
is an element of $S^{-1}((A\otimes_{\Lambda} B)[X,X^{-1}])$, where $S$
is the multiplicative set consisting of polynomials
$\sum_{r=t}^sa_rX^r$ with $a_t$ and $a_s$ being units in
$A\otimes_{\Lambda} B$. Now, let us consider the functions
$\widetilde{W}$ and $\widetilde{W'}$, defined as
$$
\widetilde{W}(g) : = W(w_n(g^t)^{-1})
$$ and
$$
\widetilde{W'}(g) : = W'(w_{n-1}(x^t)^{-1}),
$$
for all $g\in G_n(K)$ and $x\in G_{n-1}(K)$. Making change of
variables, we have the identity:
\begin{equation}\label{formulation_zeta}
c_r^K(\widetilde{W},\widetilde{W'}) = 
c_{-r}^K\big(\pi(w_n)W,\pi'(w_{n-1})W'\big).
\end{equation}
\subsubsection{Functional Equation}
Given $\pi$ and $\pi'$ as above, there is a unique element
$\gamma(X,\pi,\pi',\psi_K)$ in the ring of fraction
$S^{-1} ((A\otimes_{\Lambda} B)[X,X^{-1}])$ such that
$$
\sum_{r\in\mathbb{Z}}c_r^K(\widetilde{W},
\widetilde{W'})X^{-r}=
\varpi_{\pi'}(-1)^{n-1}\gamma(X,\pi,\pi',\psi_K)
\sum_{r\in\mathbb{Z}}c_r^K(W,W')X^r,
$$
for all $W\in\mathbb{W}(\pi,\psi_{K,A})$ and
$W'\in\mathbb{W}(\pi',\psi_{K,B}^{-1})$.

Let $e$ (resp. $e'$) be the primitive idempotent in $\mathcal{Z}_n^K$
(resp. $\mathcal{Z}_{n-1}^K$) such that the supercuspidal support map
$f_\pi$ (resp. $f_{\pi'}$) factors through the center
$e\mathcal{Z}_n^K$ (resp. $e'\mathcal{Z}_{n-1}^K$). The gamma factor
$\gamma(X,eW_n^K,e'W_{n-1}^K,\psi_K)$ corresponding to the pair
$(eW_n^K, e'W_{n-1}^K)$ admits the following universal property
(\cite[Theorem 5.4]{local_constant_Moss}) in the families of
co-Whittaker modules.
\begin{theorem}\label{univ_gamma_factors}
Let $A$ and $B$ be two Noetherian $\Lambda$-algebras. Let $\pi$ and
$\pi'$ be two co-Whittaker $A[G_n(K)]$ and $B[G_{n-1}(K)]$-modules
respectively. Let $e$ (resp. $e'$) be the primitive idempotent
of $\mathcal{Z}_n^K$ (resp. $\mathcal{Z}_{n-1}^K$) such that supercuspidal support $f_{\pi}$ (resp. $f_{\pi'}$) factors through $e\mathcal{Z}_n^K$ (resp. $e'\mathcal{Z}_{n-1}^K$). Then
$$
\gamma(X,\pi,\pi',\psi_K) = 
(f_{\pi}\otimes f_{\pi'})\big(\gamma(X,eW_n^K,e'W_{n-1}^K,\psi_K)\big).
$$
\end{theorem}

\subsection{Frobenius twist} 
Let $R$ be a Noetherian $\mathbb{F}$-algebra. Let $(\pi,V)$ be a co-Whittaker $R[G_n(K)]$-module. Let ${\rm Fr}:R\rightarrow R$ be the map $x\mapsto x^l$. We denote by  $V^{(l)}$ the $R$
module $V\otimes_{{\rm Fr}}R$. The module $(\pi^{(l)}, V^{(l)})$ is called the
Frobenius twist of $(\pi,V)$.  Let $\mathfrak{s}$ be the inertial equivalence class such that $\pi \in \mathcal{M}_{\mathfrak{s}}(G_n(K))$, and let $e$ be the primitive idempotent of $\mathcal{Z}_n^K$, associated with $\mathfrak{s}$. The composition
$$
e\mathcal{Z}_n^K\xrightarrow {f_{\pi}} R \xrightarrow{\lambda\mapsto
  \lambda ^l} R,
$$
denoted by $f_{\pi^{(l)}}$, is the supercuspidal support of the
Frobenius twist $\pi^{(l)}$. Let $R'$ be a Noetherian
$\mathbb{F}$-algebra, and let $(\pi',V')$ be the co-Whittaker
$R'[G_{n-1}(K)]$-module. Let $\mathfrak{s}'$ be the inertial equivalence class such that
$\pi'\in \mathcal{M}_{\mathfrak{s}'}(G_{n-1}(K))$, and let $e'$ be the primitive idempotent in
$\mathcal{Z}_{n-1}^K$, associated with $\mathfrak{s'}$. Using Theorem
\ref{univ_gamma_factors} and the properties of supercuspidal supports
$f_{\pi^{(l)}}$ and $f_{\pi'^{(l)}}$, we get the following identity of
gamma factors
\begin{equation}\label{mod_l_gamma_identity}
\gamma(X,\pi,\pi',\psi_K)^l =  \gamma(X^l,\pi^{(l)},\pi'^{(l)},\psi_K^l).
\end{equation}

In the next subsection, we set up an idenitity of the integrals on the
homeogeneous space of $F$ with those on the homogeneous space of
$E$. This is crucial for the main results.
\subsection{Integrals on homogeneous space}
As before, let $E$ be a finite Galois extension of a $p$-adic field
$F$ with degree of extension $l$, where $l$ is coprime to $p$. We
denote by $\Gamma$ the Galois group ${\rm Gal}(E/F)$. Let $R$ be an
$\mathbb{F}$-algebra with $p$ being invertible $R$. There is a natural
action of $\Gamma$ on the coset space $X_E$ and hence on the space
$C_c^\infty(X_E,R)$, consisting of smooth and compactly supported
$R$-valued functions on $X_E$, given by
$(\gamma.\varphi)(g):=\varphi(\gamma^{-1}g)$, for all
$\gamma\in\Gamma$, $f\in C_c^\infty(X_E,R)$ and for all $g\in X_E$. We
denote by $C_c^\infty(X_E,R)^\Gamma$ the space of $\Gamma$-fixed
elements in $C_c^\infty(X_E,R)$. With these, we now deduce:
\begin{proposition}
Let $d\mu_E$ and $d\mu_F$ be the Haar measures on $X_E$ and $X_F$
respectively. Then, there exists a non-zero scalar
$\alpha\in \mathbb{F}$ such that for all
$\varphi \in C_c^\infty(X_E,R)^\Gamma$, we have
$$ \int_{X_E}\varphi \,d\mu_E = \alpha \int_{X_F}\varphi \,d\mu_F. $$
\end{proposition}
\begin{proof}
  The proof is immediate by following the arguments of
  \cite[Proposition 5.2]{dhar2022tate} mutatis-mutandis.
\end{proof}
\begin{remark}\label{rmk_int2}
\rm The Haar measures $d\mu_E$ and $d\mu_F$ on $X_E$ and $X_F$,
respectively, are now choosen in a way that ensures $\alpha = 1$.
Moreover, if $e$ is the ramification index of the extension $E$ over
$F$, then for all $k\notin\{re : r\in \mathbb{Z}\}$, we have
$$ \int_{(X_E^k)^\Gamma}\varphi\,d\mu_F = 0, $$
and for all $k\in\{re : r\in \mathbb{Z}\}$, we have
$$ \int_{(X_E^k)^\Gamma}\varphi\,d\mu_F =
\int_{X_F^{\frac{r}{e}}}\varphi\,d\mu_F. $$
\end{remark}

\section{Tate cohomology of co-Whittaker modules}
In this section, we study the compatibility of Tate cohomology with
local base change for universal co-Whittaker modules.
\subsection{}
Let $E$ be a finite Galois extension of a $p$-adic field $F$ with
$[E:F]=l$, where $l$ and $p$ are distinct primes. We fix a generator
$\gamma$ of the cyclic group ${\rm Gal}(E/F)$. For any
$\Lambda$-algebra $R$ and any smooth $R[G_E]$-module $V$, we denote by
$V^\gamma$ the smooth $R[G_E]$-module, where the underlying set is $V$
but the $G_E$-action on $V^\gamma$ is twisted by $\gamma$. Note that
the functor $V\mapsto V^\gamma$ is exact and covariant.

\subsection{Galois invariance}
Let $e_{\mathfrak{s}}$ be a primitive idempotent of $\mathcal{Z}_n^F$,
and let $\nu$ be the $l$-inertial type corresponding to
$e_{\mathfrak{s}}$. Then there exists a primitive idempotent
$e_{\mathfrak{r}}$ of $\mathcal{Z}_n^E$, which corresponds to the same
$l$-inertial type $\nu$. We have the base change maps
$z_{E/F}:e_{\mathfrak{r}}\mathcal{Z}_n^E\rightarrow
e_{\mathfrak{s}}\mathcal{Z}_n^F$ and
$B_{E/F}:R_\nu^E\rightarrow R_\nu^F$, as defined in the subsection
(\ref{base_change_map}), with the following identity
\begin{equation}\label{commutative_1}
\mathbb{L}_\nu^F \circ z_{E/F} = B_{E/F} \circ \mathbb{L}_\nu^E.
\end{equation}
We denote by $A_E$ and $A_F$ the Noetherian $\Lambda$-algebras
$e_{\mathfrak{r}}\mathcal{Z}_n^E$ and
$e_{\mathfrak{s}}\mathcal{Z}_n^E$ respectively. Note that the ring
$A_F$ is considered as an $A_E$-module via the map $z_{E/F}$.

\begin{lemma}\label{inv_1}
Let $\mathcal{V}$ be the co-Whittaker $A_F[G_n(E)]$-module
$e_{\mathfrak{r}}W_n^E\otimes_{A_E} A_F$. Then we have
$$
\mathbb{W}(\mathcal{V},\psi_{E,A_F})
= \mathbb{W}(\mathcal{V}^\gamma,\psi_{E,A_F}).
$$
\end{lemma}
\begin{proof}
We use the local converse therorem (\cite[Theorem
1.1]{Converse_thm_Moss_Liu}) for co-Whittaker modules.
It is enough to show that 
$$
\gamma(X,\mathcal{V},\tau,\psi_E) = 
\gamma(X,\mathcal{V}^\gamma,\tau,\psi_E),
$$
for $\tau$ varying over all irreducible generic integral
$\overline{\mathcal{K}}$-representation of $G_t(E)$, for
$1\leq t \leq [\frac{n}{2}]$.  For a specialization
$x:A_F\rightarrow \overline{\mathcal{K}}$, the $G_n(E)$ representation
$(e_\mathfrak{r}W_n^E\otimes_{A_E}A_F)\otimes_x
\overline{\mathcal{K}}$, denoted by $\Pi_x$, is the base change lift
of a smooth representation $\pi_x$ of $G_n(F)$. In particular,
$\Pi_x^\gamma\simeq \Pi_x$, for all $\gamma\in \Gamma$. Since
$$x(\gamma(X,\mathcal{V},\tau,\psi_E))
=\gamma(X,\Pi_x,\tau,\psi_E),$$
and 
$$x(\gamma(X,\mathcal{V}^\gamma,\tau,\psi_E))
=\gamma(X,\Pi_x^\gamma,\tau,\psi_E),$$
we have
$$x(\gamma(X,\mathcal{V},\tau,\psi_E)) = 
x(\gamma(X,\mathcal{V}^\gamma,\tau,\psi_E)),$$ for all
$\overline{\mathcal{K}}$-points of $A_F$.  The lemma now follows since
$A_F$ is reduced and flat over $\Lambda$.
\end{proof}
From these, we deduce:
\begin{lemma}\label{inv_whit_model}
Let $\mathcal{V}$ be the co-Whittaker module
$e_{\mathfrak{r}}W_n^E\otimes_{A_E} A_F$. Then the Whittaker space
$\mathbb{W}(\mathcal{V}, \psi_{E,A_F})$ is invariant under the
action of ${\rm Gal}(E/F)$.
\end{lemma}
\begin{proof}
  Since $\mathcal{V}$ is co-Whittaker, we have the $A_F$-module
  isomorphism $\mathcal{V}^{(n)} \simeq A_F$. Precomposing this
  isomorphism with the quotient map
  $\mathcal{V}\rightarrow \mathcal{V}^{(n)}$ induces the $A_F$-linear
  map $\mathcal{W}:\mathcal{V}\rightarrow A_F$ with
$$
\mathcal{W}(n.v) = \Psi_{E,A_F}(n)\mathcal{W}(v),
$$
for all $v \in \mathcal{V}$ and $n\in N_n(E)$. Moreover, for any
$\gamma \in {\rm Gal}(E/F)$, we have
$$
\mathcal{W}(\gamma(n).v) = \Psi_{E,A_F}(\gamma(n))\mathcal{W}(v) =
\Psi_{E,A_F}(n)\mathcal{W}(v).
$$ 
Therefore, the Whittaker spaces of $\mathcal{V}$ and
$\mathcal{V}^\gamma$ with respect to the character $\psi_E$ are
induced by the same linear map $\mathcal{W}$. Let $W_v$ be an element
of $\mathbb{W}(\mathcal{V},\psi_{E,A_F})$. Then
$$
(\gamma^{-1}.W_v)(g) = \mathcal{W}(\gamma(g).v).
$$ 
This shows that
$\gamma^{-1}.W_v \in \mathbb{W}(\mathcal{V}^\gamma,\psi_{E,A_F})$. But
$\mathbb{W}(\mathcal{V}^\gamma,\psi_{E,A_F}) =
\mathbb{W}(\mathcal{V},\psi_{E,A_F})$, by Lemma \ref{inv_1}. This
completes the proof.
\end{proof}
We now prove the following theorem.
\begin{theorem}\label{Tate_coWhittaker}
Let $\mathcal{V}$ be the co-Whittaker $A_F[G_n(E)]$-module
$e_{\mathfrak{r}}W_n^E\otimes_{A_E} A_F$, and let $\mathcal{V}_F$ be
the $\mathbb{F}[G_n(F)]$-module
$e_{\mathfrak{s}}W_n^F\otimes_{\Lambda} \mathbb{F}$. Let $R$ be the
$\mathbb{F}$-algebra $A_F\otimes_{\Lambda}\mathbb{F}$. Then there is
a $G_n(F)$-stable subspace $\mathcal{M}$ of the Tate cohomology
group $\widehat{H}^0(\mathbb{W}(\mathcal{V},\psi_{E,A_F}))$
with $G_n(F)$-equivariant surjection
$$
\mathcal{M}\longrightarrow \mathbb{W}(\mathcal{V}_F^{(l)},\psi_{F,R}^l).
$$
\end{theorem}
\begin{proof}
We prove the above theorem using induction on the integer $n$.
\subsubsection{}
Let us consider the case $n=1$. Recall that $\mathcal{Z}_1^K$ is the
convolution $\Lambda$-algebra $C^\infty_c(K^\times, \Lambda)$. Let
$P_K$ be the prime to $l$ part of $\mathfrak{o}_K^\times$.  For any
character $\eta:P_K\rightarrow \Lambda^\times$, we get an idempotent
$e_\eta$ of $\mathcal{Z}_1^K$, and the $\Lambda$-algebra
$e_\eta\mathcal{Z}_1^K$ is equal to
$\mathcal{H}(K^\times, P_K, \eta)$--the endomorphism algebra of the
$\Lambda[E^\times]$ representation $\ind_{P_K}^{K^\times}\eta$. Local
class field theory gives a character $\nu$ of $I_K^{(l)}$ assocaited
with $\eta$.  The ring $R_\nu^K$ is isomorphic to
$\Lambda[E^\times/P_K]$ and $\mathbb{L}_\eta^K$ is the identity map.
Let $\chi:F^\times\rightarrow \mathbb{F}$ be a character and let
$\tilde{\chi}$ be the character $\chi\circ {\rm Nr}_{E/F}$.  The
characters $\chi$ and $\tilde{\chi}$ give rise to idempotents
$e_\mathfrak{s}$ and $e_\mathfrak{r}$ of $\mathcal{Z}_1^F$ and
$\mathcal{Z}_1^E$ respectively.  From local class field theory the map
$B_{E/F}$ is induced by the norm map
${\rm Nr}_{E/F}:E^\times \rightarrow F^\times$. The module
$e_{\mathfrak{r}}\mathcal{Z}_1^E$ also has a natural $E^\times$-action
and the identification
$$e_{\mathfrak{r}}\mathcal{Z}_1^E
\otimes_{z_{E/F}} e_{\mathfrak{s}}\mathcal{Z}_1^F\simeq
e_{\mathfrak{s}}\mathcal{Z}_1^F,$$ makes
$e_{\mathfrak{s}}\mathcal{Z}_1^F$ an
$e_{\mathfrak{s}}\mathcal{Z}_1^F[E^\times]$-module.  Thus, we have
$$
\widehat{H}^0(e_{\mathfrak{s}}\mathcal{Z}_1^F)\xrightarrow{\sim}
(e_{\mathfrak{s}}\mathcal{Z}_1^F\otimes_\Lambda \mathbb{F})^{(l)}.
$$ 

\subsubsection{}
So now, we assume that the result is true for $n-1$. We denote by
$\tau_E$ (resp. $\tau_F$) the action of $G_n(E)$ (resp. $G_n(F)$) on
the space $\mathcal{V}$ (resp. $\mathcal{V}_F$). Recall that
$\mathbb{W}(\mathcal{V},\psi_{E,A_F})$ is invariant under
$\Gamma$-action. Let $\Phi_n$ be the composite map
$$
\mathbb{K}(\mathcal{V},\psi_{E,A_F})^\Gamma \xrightarrow{\theta_\ell} 
{\rm Ind}_{N_n(E)}^{P_n(E)}(\Psi_{E,R}) \xrightarrow{{\rm res}_{P_n(F)}}
{\rm Ind}_{N_n(F)}^{P_n(F)}(\Psi_{F,R}^l),
$$
where the map $\theta_\ell$ is induced by the morphism
$A_F\rightarrow R$, sending $x$ to $x\otimes 1$. Note that the map
$\Phi_n$ factorizes through the Tate cohomology space
$\widehat{H}^0(\mathbb{K}(\mathcal{V},\psi_{E,A_F}))$. Let
$\mathcal{M}(\psi_F)$ be the inverse image of the Kirillov space
$\mathbb{K}(\mathcal{V}_F^{(l)},\psi_{F,R}^l)$ under $\Phi_n$. Note
that $\mathcal{M}(\psi_F)$ is non-zero, and it is stable under the
action of $P_n(F)$ with non-zero $P_n(F)$-equivariant map
$$
\Phi_n:\mathcal{M}(\psi_F)\longrightarrow \mathbb{K}(\mathcal{V}_F^{(l)},\psi_{F,R}^l).
$$
Let $V$ be an element of $\mathcal{M}(\psi_F)$. We will
show that
\begin{equation}\label{claim_1}
\Phi_n(\overline{\tau_E(w_n)}V)=\tau_F^{(l)}(w_n)\Phi_n(V),
\end{equation}
where $\overline{\tau_E(w_n)}$ is the induced action of $\tau_E(w_n)$
on $\widehat{H}^0(\mathbb{K}(\mathcal{V},\psi_{E,A_F}))$.
\subsubsection{}
Let $e_{\mathfrak{s}'}$ be an arbitrary primitive idempotent of
$\mathcal{Z}_{n-1}^F$ and let $\nu'$ be the $l$-inertial type
corresponding to $e_{\mathfrak{s}'}$. Then there exists a primitive
idempotent $e_{\mathfrak{r}'}$ of $\mathcal{Z}_{n-1}^E$, which
corresponds to the same $l$-inertial type $\nu'$. If $A_E'$
(resp. $A_F'$) denotes the Noetherian $\Lambda$-algebras
$e_{\mathfrak{r}'}\mathcal{Z}_{n-1}^E$
(resp. $e_{\mathfrak{s}'}\mathcal{Z}_{n-1}^F$), then we have the base
change map $z_{E/F}':A_E'\rightarrow A_F'$ with
\begin{equation}\label{commutative_2}
\mathbb{L}_{\nu'}^F\circ z_{E/F}' = B_{E/F}'\circ \mathbb{L}_{\nu'}^E.
\end{equation}
We denote by $\mathcal{V}'$ the co-Whittaker $A_F'[G_{n-1}(E)]$-module
$e_{\mathfrak{r}}'W_{n-1}^E\otimes_{A_E'} A_F'$. Let $R'$ be the
$\mathbb{F}$-algebra $A_F'\otimes_{\Lambda}\mathbb{F}$, and let
$\mathcal{V}_F'$ the co-Whittaker $R'[G_{n-1}(F)]$-module
$e_{\mathfrak{s}'} W_{n-1}^F\otimes_{\Lambda}\mathbb{F}$.  Using
induction hypothesis, there is a $G_{n-1}(F)$-stable subspace
$\mathcal{M}'(\psi_F)$ of the Tate cohomology group
$\widehat{H}^0(\mathbb{W}(\mathcal{V}',\psi_{E,A_F'}^{-1}))$ such that
the non-zero $P_n(F)$-equivariant map
$$
\Phi_{n-1} : \mathbb{K}(\mathcal{V}',\psi_{E,A_F'}^{-1})^\Gamma
\xrightarrow {\theta_\ell'}{\rm
  Ind}_{N_{n-1}(E)}^{P_{n-1}(E)}(\Psi_{E,R'}^{-1})\xrightarrow{{\rm
    res}_{P_{n-1}(F)}} {\rm
  Ind}_{N_{n-1}(F)}^{P_{n-1}(F)}(\Psi_{F,R'}^{-l}),
$$ 
gives the $G_n(F)$-equivariant surjection
$$
\mathcal{M}'(\psi_F) \longrightarrow
\mathbb{W}(\mathcal{V}_F',\psi_{F,R'}^{-l}),
$$
which we also denote by $\Phi_{n-1}$. Here, the map $\theta_\ell'$ is
induced by the morphism $A_F'\rightarrow R'$, sending $\alpha$ to
$\alpha\otimes 1$. We denote by $\tau_E'$ (resp. $\tau_F'$) the action
of $G_{n-1}(E)$ (resp. $G_{n-1}(F)$) on the space $\mathcal{V}'$
(resp. $\mathcal{V}_F'$).

\subsubsection{}
Let $V'$ be an element of
$\mathbb{W}(\mathcal{V}_F'^{(l)},\psi_{F,R'}^{-l})$. Then there exists
an element $W'$ in the Whittaker space
$\mathbb{W}(\mathcal{V}',\psi_{E,A_F'})^\Gamma$ such that
${\rm res}_{P_{n-1}(E)}(W')$ is mapped to ${\rm res}_{P_{n-1}(F)}(V')$
under $\Phi_{n-1}$. Let $W$ be the element of
$\mathbb{W}(\mathcal{V},\psi_{E,A_F})^\Gamma$ such that
${\rm res}_{P_n(E)}(W) = V$. Then, using functional equation over $E$,
we get
$$
\sum_{k\in\mathbb{Z}}c_k^E(\widetilde{W}, \widetilde{W'})X^{-fk} =
\varpi_{\tau_E'}(-1)^{n-1}\gamma(X,\mathcal{V},\mathcal{V}',\psi_E)
\sum_{k\in\mathbb{Z}}c_k^E(W,W')X^{fk},
$$
where $f$ is the residue degree of the extension $E/F$. Applying the
morphism $(\theta_\ell\otimes\theta_\ell')$, and using the identity
(\ref{formulation_zeta}) and Remark \ref{rmk_int2}, the above relation
becomes
\begin{equation}\label{fe_1}
\begin{split}
  \sum_{k\in\mathbb{Z}}c_{-k}^F\big(\overline{\tau_E(w_n)}
  &\theta_\ell(W),\overline
  {\tau_E'(w_{n-1})}\theta_\ell'(W')\big)X^{-lk}\\
  &= \varpi_{\tau_E'}(-1) ^{n-1} (\theta_\ell
  \otimes\theta_\ell')(\gamma(X,\mathcal{V},\mathcal{V}',\psi_E))
  \sum_{k\in\mathbb{Z}} c_k^F\big(\theta_\ell(W),V'\big)X^{lk}.
\end{split}
\end{equation}
By induction hypothesis, we have
$$ \Phi_{n-1}(\overline{\tau_E'(w_{n-1})}W') = 
\tau_F'^{(l)}(w_{n-1})V'.
$$
Using this, it follows from (\ref{fe_1}) that
\begin{equation}\label{fe_2}
\begin{split}
  \sum_{k\in\mathbb{Z}}c_{-k}^F\big(\overline{\tau_E(w_n)}&\theta_\ell(W),
  \tau_F'^{(l)}(w_{n-1})V'\big)X^{-lk} \\
  &= \varpi_{\tau_F'}(-1) ^{l(n-1)} (\theta_\ell
  \otimes\theta_\ell')(\gamma(X,\mathcal{V},\mathcal{V}',\psi_E))
  \sum_{k\in\mathbb{Z}} c_k^F(\theta_\ell(W),V')X^{lk}.
\end{split}
\end{equation}
\subsubsection{}
Recall that $\Phi_n(V)$ is an element of the Kirillov space
$\mathbb{K}(\mathcal{V}_F^{(l)},\psi_{F,R}^l)$. Let $U$ be the element
of $\mathbb{W}(\mathcal{V}_F^{(l)},\psi_{F,R}^l)$ such that
${\rm res}_{P_n(F)}(U)=\Phi_n(V)$.  By Theorem \ref{van_result}, the
assertion (\ref{claim_1}) is equivalent to the following identity:
\begin{equation}\label{claim_2}
\begin{split}
  \sum_{k\in\mathbb{Z}}c_{-k}^F\big(\Phi_n(\overline{\tau_E(w_n)}\theta_\ell(W)),
  \tau_F'^{(l)}&(w_{n-1})V'\big)X^{-lk}\\ =&
  \sum_{k\in\mathbb{Z}}c_{-k}^F\big(\tau_F^{(l)}(w_n)U,
  \tau_F'^{(l)}(w_{n-1})V'\big)X^{-lk}
\end{split}
\end{equation}
From functional equation over F, we get
\begin{align*}
  \sum_{k\in\mathbb{Z}}c_k^F(\widetilde{U},
  \widetilde{V'})X^{-k}=
  \varpi_{\tau_F'}(-1)^{l(n-1)}\gamma(X,\mathcal{V}_F^{(l)},\mathcal{V}_F'^{(l)},\psi_F^l)
  \sum_{k\in\mathbb{Z}}c_k^F(U,V')X^k.
\end{align*}
Using the relation (\ref{formulation_zeta}) and replacing the variable
$X$ by $X^l$ to the above equality, we have
\begin{align*}
\begin{split}
  \sum_{k\in\mathbb{Z}}c_{-k}^F\big(\tau_F^{(l)}(w_n)U,\ &
  \tau_F'^{(l)}(w_{n-1})V'\big)X^{-lk} \\
  = &\varpi_{\tau_F'}(-1) ^{l(n-1)}
  \gamma(X,\mathcal{V}_F^{(l)},\mathcal{V}_F'^{(l)},\psi_F^l)
  \sum_{k\in\mathbb{Z}}c_k^F(U,V')X^{lk}.
\end{split}    
\end{align*}
\subsubsection{}
Comparing the above equation with (\ref{fe_2}), the equality
(\ref{claim_2}) is now equivalent to the following identity of gamma
factors
$$
(\theta_\ell
\otimes\theta_\ell')(\gamma(X,\mathcal{V},\mathcal{V}',\psi_E)) =
\gamma(X,\mathcal{V}_F^{(l)},\mathcal{V}_F'^{(l)}, \psi_F^l).
$$
First, note that the base change maps $B_{E/F}$ and $B_{E/F}'$
together with the universal property of the pairs
$(R_\nu^E,\rho_\nu^E)$ and $(R_{\nu'}^E,\rho_{\nu'}^F)$ induces the
isomorphisms
\begin{center}
  ${\rm res}_{\mathcal{W}_E}(\rho_\nu^F)\simeq
  \rho_\nu^E\otimes_{R_\nu^E}R_\nu^F$ and
  ${\rm res}_{\mathcal{W}_E}(\rho_{\nu'}^F)\simeq
  \rho_{\nu'}^E\otimes_{R_{\nu'}^E}R_{\nu'}^F$.
\end{center}
Using these isomorphisms and the commutativity relations
(\ref{commutative_1}) and (\ref{commutative_2}), we have
\begin{align*}
  \gamma(X,\mathcal{V},\mathcal{V}',\psi_E) 
  &= (z_{E/F}\circ(\mathbb{L}_\nu^E)^{-1})\otimes
    (z_{E/F}'\circ(\mathbb{L}_{\nu'}^E)^{-1})\big(\gamma(X,\rho_\nu^E\otimes
    \rho_{\nu'}^E,\psi_E)\big)\\
  &= ((\mathbb{L}_\nu^F)^{-1}\circ B_{E/F})\otimes ((\mathbb{L}_{\nu'}^F)^{-1}
    \circ B_{E/F}')(\gamma(X,\rho_{\nu}^E\otimes\rho_{\nu'}^E,\psi_E))\\
  &= ((\mathbb{L}_\nu^F)^{-1}\otimes(\mathbb{L}_{\nu'}^F)^{-1})
    \big(\gamma(X,(\rho_\nu^E\otimes_{R_\nu^E}R_\nu^F)\otimes
    (\rho_{\nu'}^E\otimes_{R_{\nu'}^E}R_{\nu'}^F),\psi_E)\big)\\
  &= ((\mathbb{L}_\nu^F)^{-1}\otimes(\mathbb{L}_{\nu'}^F)^{-1})
    \big(\gamma(X,\rho_\nu^F\otimes\rho_{\nu'}^F\otimes
    {\rm ind}_{\mathcal{W}_E}^{\mathcal{W}_F}(1_E),\psi_F)\big)\\
  &= \prod_{\eta}\gamma(X,e_{\mathfrak{s}}W_n^F,
    e_{\mathfrak{s}'}W_{n-1}^F\otimes\eta,\psi_F),
\end{align*}
where $\eta$ runs over the characters of ${\rm Gal}(E/F)$. Finally,
applying the morphism $(\theta_\ell \otimes\theta_\ell')$ and using
the identity (\ref{mod_l_gamma_identity}), we get
$$
(\theta_\ell \otimes \theta_\ell')(\gamma(X,\mathcal{V},\mathcal{V}',\psi_E)) =
\gamma(X^l,\mathcal{V}_F^{(l)},\mathcal{V}_F'^{(l)},
\psi_F^l).
$$
This shows that $\Phi_n$ is a non-zero $G_n(F)$-equivariant map. Since
$\mathcal{V}_F^{(l)}$ is co-Whittaker $R[G_n(F)]$-module, the map
$\Phi_n$ is infact a surjection. This completes the proof.
\end{proof}

\section{Tate cohomology of generic representations}
In this section, we prove our main result (Theorem
\ref{intro_main_thm}) using Theorem \ref{Tate_coWhittaker}. We will
continue with the notations of the preceding section.
\subsection{}
Let $\pi_F$ be an integral generic representation of $G_n(F)$ over a
$\overline{\mathcal{K}}$-vector space whose mod-$l$ inertial
supercuspidal support is $\mathfrak{s}$. Let $\pi_E$ be the base
change lifting of $\pi_F$ with mod-$l$ inertial supercuspidal support
$\mathfrak{r}$. Recall that we have the base change map
$$
z_{E/F} : e_{\mathfrak{r}}\mathcal{Z}_n^E \longrightarrow
e_{\mathfrak{s}}\mathcal{Z}_n^F.
$$
If $f_{\pi_F}$ (resp. $f_{\pi_E}$) denotes the supercuspidal support
of $\pi_F$ (resp. $\pi_E$), then we have
$$
f_{\pi_E} = f_{\pi_F} \circ z_{E/F}.
$$
Let $J_\ell(\pi_F)$ be the unique generic sub-quotient of the mod-$l$
reduction $r_\ell(\pi_F)$. The supercuspidal support of the
$\mathbb{F}$-representation $J_\ell(\pi_F)$ is equal to $\mathfrak{s}$
(\cite[Proposition 4.13]{Helm_Bernstein_center}). We denote by $A_E$
and $A_F$ the Noetherian $\Lambda$-algebras
$e_{\mathfrak{r}}\mathcal{Z}_n^E$ and
$e_{\mathfrak{s}}\mathcal{Z}_n^F$ respectively.
\begin{theorem}
Let $F$ be a $p$-adic field and let $E$ be a finite Galois extension
of $F$ with $[E:F]=l$, where $l$ and $p$ are distinct odd
primes. Let $\pi_F$ be an integral generic
$\overline{\mathcal{K}}$-representation of $G_n(F)$, and let $\pi_E$
be the base change lifting of $\pi_F$ to $G_n(E)$. Let
$\mathbb{W}^0(\pi_E,\psi_{E,\overline{\mathcal{K}}})$ be the
space of all $\mathcal{O}$-valued functions in
$\mathbb{W}(\pi_E,\psi_{E,\overline{\mathcal{K}}})$. Then the
$\mathbb{F}$-representation $J_\ell(\pi_F)^{(l)}$ is a subquotient
of the Tate cohomology group
$\widehat{H}^0(\mathbb{W}^0(\pi_E,\psi_{E,
\overline{\mathcal{K}}}))$.
\end{theorem}
\begin{proof}
The proof relies on the completeness of Whittaker models (Theorem
\ref{van_result}). Although the proof follows from the same line of
arguments as in Theorem \ref{Tate_coWhittaker}, we provide it in
detail for the purpose of the completeness. We follow the same
notations and terminologies as in Theorem \ref{Tate_coWhittaker}.
\subsubsection{}
Note that the lattice
$\mathbb{W}^0(\pi_E,\psi_{E,\overline{\mathcal{K}}})$ is stable
under the action of $G_n(E)$
(\cite[Theorem]{Vigneras_integral_structure}). Let $\Phi_n$ be the
composite map
$$
\mathbb{K}^0(\pi_E,\psi_{E,\overline{\mathcal{K}}})^\Gamma
\xrightarrow{r_\ell}{\rm
  Ind}_{N_n(E)}^{P_n(E)}(\Psi_{E,\mathbb{F}}^l)\xrightarrow{{\rm
    res}_{P_n(F)}} {\rm
  Ind}_{N_n(F)}^{P_n(F)}(\Psi_{F,\mathbb{F}}^l),
$$
where $r_\ell$ denotes the pointwise mod-$l$ reduction. It is clear that the map
$\Phi_n$ factorizes through the space
$\widehat{H}^0(\mathbb{K}^0(\pi_E,\psi_{E,\overline{\mathcal{K}}}))$. Let
$\mathcal{N}(\psi_F)$ be the inverse image of the Kirillov space
$\mathbb{K}(J_\ell(\pi_F)^{(l)},\psi_{F,\mathbb{F}}^l)$ under
$\Phi_n$. Then $\mathcal{N}(\psi_F)$ is a non-zero $P_n(F)$-stable
subspace of
$\widehat{H}^0(\mathbb{K}^0(\pi_E,\psi_{E,\overline{\mathcal{K}}}))$
with a non-zero $P_n(F)$-equivariant map
$$
\Phi_n:\mathcal{N}(\psi_F)\longrightarrow
\mathbb{K}(J_\ell(\pi_F)^{(l)},\psi_{F,\mathbb{F}}^l).
$$
We will prove that $\mathcal{N}(\psi_F)$ is $G_n(F)$-stable and the map
$\Phi_n$ is $G_n(F)$-equivariant. To be precise, let $V$ be an
element of $\mathcal{N}(\psi_F)$. We will show that
$$
\Phi_n(\overline{\pi_E(w_n)}V) = J_\ell(\pi_F)^{(l)}(w_n)\Phi_n(V),
$$
where $\overline{\pi_E(w_n)}$ is the linear operator on
$\widehat{H}^0(\mathbb{K}(\pi_E,\psi_{E,\overline{\mathcal{K}}}))$
induced by $\pi_E(w_n)$. Let $W$ be an element of
$\mathbb{W}^0(\pi_E,\psi_{E,\overline{\mathcal{K}}})^\Gamma$ with ${\rm res}_{P_n(E)}(W) =
V$. By Theorem \ref{van_result}, the above assertion is equivalent to
the following identity:
\begin{equation}\label{claim_3}
\sum_{k\in\mathbb{Z}} c_k^F\big(\Phi_n(\overline{\pi_E(w_n)}V), 
V'\big)X^k = \sum_{k\in\mathbb{Z}} c_k^F\big(J_\ell(\pi_F)^{(l)}(w_n)\Phi_n(V), V'\big)X^k,
\end{equation}
for all
$V'\in\mathbb{W}\big((e_{\mathfrak{s}'}W_{n-1}^F\otimes_{\Lambda}\mathbb{F})^{(l)},
\psi_{F,R'}^{-l}\big)$ and for all primitive idempotents
$e_{\mathfrak{s}'}$ of $\mathcal{Z}_{n-1}^F$. Here $R'$ denotes the ring
$e_{\mathfrak{s}'}\mathcal{Z}_{n-1}^F\otimes_{\Lambda}\mathbb{F}$.
\subsubsection{}
Let $e_{\mathfrak{s}'}$ be a primitive idempotent of
$\mathcal{Z}_{n-1}^F$. Then there exists a primitive idempotent
$e_{\mathfrak{r}'}$ of $\mathcal{Z}_{n-1}^E$ such that we have the
base change map $z_{E/F}' : A_E'\rightarrow A_F'$ with
$$
\mathbb{L}_{\nu'}^F\circ z_{E/F}' = B'_{E/F} \circ \mathbb{L}_{\nu'}^E,
$$
where $\nu'$ is the $l$-inertial type, which corresponds to both
$e_{\mathfrak{r}'}$ and $e_{\mathfrak{s}'}$. As before, we denote by
$\mathcal{V}'$ and $\mathcal{V}_F'$ the co-Whittaker modules
$e_{\mathfrak{r}'}W_{n-1}^E\otimes_{A_E'} A_F'$ and
$e_{\mathfrak{s}'}W_{n-1}^F \otimes_{\Lambda} \mathbb{F}$
respectively. The action of $G_n(E)$ (resp. $G_n(F)$) on the space
$\mathcal{V}'$ (resp. $\mathcal{V}_F'$) is denoted by $\pi_E'$
(resp. $\pi_F'$).
\subsubsection{}
Let $V'$ be an element of
$\mathbb{W}(\mathcal{V}_F'^{(l)}, \psi_{F,R'}^{-l})$. Following
the arguments of Theorem \ref{Tate_coWhittaker}, we get an element
$W'$ in $\mathbb{W}(\mathcal{V}',\psi_{E,A_F'}^{-1})^\Gamma$ such
that $W'$ is mapped to $V'$ under the composition
$$
\Phi_{n-1} : \mathbb{W}(\mathcal{V}',\psi_{E,A_F'}^{-1})^\Gamma
\xrightarrow{\theta_\ell'} {\rm
  Ind}_{N_{n-1}(E)}^{P_{n-1}(E)}(\Psi_{E,R'}^{-1}) \xrightarrow
{{\rm res}_{G_{n-1}(F)}} {\rm
  Ind}_{N_{n-1}(F)}^{G_{n-1}(F)}(\Psi_{F,R'}^{-l}),
$$
with the following identity
\begin{equation}\label{induction_equ}
\Phi_{n-1}(\pi_E'(w_{n-1})W') = \pi_F'^{(l)}(w_{n-1})V'.
\end{equation}
Recall that the map $\theta_\ell'$ is induced by the morphism $A_F'\rightarrow R'$,
sending $x$ to $x\otimes 1$. From functional equation over $E$, we get
\begin{align*}
    \sum_{k\in\mathbb{Z}} c_{-k}^E\big(\overline{\pi_E(w_n)}&r_l(W),
\pi_E'(w_{n-1})\theta_\ell'(W')\big)X^{-fk} \\
&=\varpi_{\pi_E'}(-1)^{n-1}(r_\ell\otimes\theta_\ell')\big(\gamma(X,\pi_E,\mathcal{V}',
\psi_E)\big) \sum_{k\in\mathbb{Z}} c_k^E\big(r_l(W),V'\big)X^{fk}.
\end{align*}
Using Remark \ref{rmk_int2} and the relation (\ref{induction_equ}), it
follows from the above identity that
\begin{equation}\label{integral_FE}
\begin{split}
 \sum_{k\in\mathbb{Z}} c_{-k}^F\big(\overline{\pi_E(w_n)}&r_l(W),
\pi_F'^{(l)}(w_{n-1})V'\big)X^{-lk} \\
&=\varpi_{\pi_F'}(-1)^{l(n-1)}(r_\ell\otimes\theta_\ell')\big(\gamma(X,\pi_E,\mathcal{V}',
\psi_E)\big) \sum_{k\in\mathbb{Z}} 
c_k^F(r_l(W),V')X^{lk}.   
\end{split}
\end{equation}
\subsubsection{}
Let $U$ be an element of
$\mathbb{W}(J_\ell(\pi_F)^{(l)},\psi_{F,\mathbb{F}}^{-l})$ such
that ${\rm res}_{P_n(F)}(U) = \Phi_n(V)$. The functional equation over
$F$ gives
\begin{align*}
\begin{split}
 \sum_{k\in\mathbb{Z}} c_{-k}^F\big(J_\ell(\pi_F)^{(l)}&(w_n)U,
\pi_F'^{(l)}(w_{n-1})V'\big)X^{-k}\\
&= \varpi_{\pi_F'}(-1)
^{l(n-1)}\gamma(X,J_\ell(\pi_F)^{(l)},
\mathcal{V}_F'^{(l)}, \psi_F^l) \sum_{k\in\mathbb{Z}} c_k^F(U,V')X^k   \end{split}
\end{align*}
Replacing $X$ by $X^l$ to the above equation, we get
\begin{equation}\label{FE_2}
\begin{split}
 \sum_{k\in\mathbb{Z}} c_{-k}^F\big(J_\ell(\pi_F)^{(l)}&(w_n)U,
  \pi_F'^{(l)}(w_{n-1})W'\big)X^{-lk} \\
  &= \varpi_{\pi_F'}(-1)
  ^{l(n-1)} \gamma(X^l,J_\ell(\pi_F)^{(l)},\mathcal{V}_F'^{(l)},
  \psi_F^l) \sum_{k\in\mathbb{Z}} c_k^F(U,W')X^{lk}.   
\end{split}
  \end{equation}
\subsubsection{}
Comparing the relations (\ref{integral_FE}) and (\ref{FE_2}), the
assertion (\ref{claim_3}) is now equivalent to the following identity
of gamma factors
$$
(r_\ell\otimes\theta_\ell')\big(\gamma(X,\pi_E,
\mathcal{V}',\psi_E)\big) =
\gamma(X^l,J_\ell(\pi_F)^{(l)},
\mathcal{V}_F'^{(l)},\psi_F^l).
$$
This follows from similar type of computation of gamma factors as we
did in Theorem \ref{Tate_coWhittaker}. First, note that
$$
\gamma(X,\pi_E,\mathcal{V}',\psi_E) = 
(f_{\pi_E}\otimes z_{E/F}')(\gamma(X,e_{\mathfrak{r}}W_n^E,e_{\mathfrak{r}'}W_{n-1}^E,\psi_E)).
$$
Using the relations (\ref{commutative_1}) and (\ref{commutative_2})
and the fact that $f_{\pi_E} = f_{\pi_F}\circ z_{E/F}$, we get the
following equality:
$$
\gamma(X,\pi_E,\mathcal{V}',\psi_E) 
= \prod_{\eta} \gamma(X,\pi_F,e_{\mathfrak{s}'}
W_{n-1}^F\otimes 
\eta,\psi_F),
$$
where $\eta$ runs over the characters of ${\rm Gal}(E/F)$. Applying
the morphism $(r_\ell\otimes \theta_\ell')$ to the above identity, we
get
$$
(r_\ell\otimes \theta_\ell')(\gamma(X,\pi_E,\mathcal{V}',\psi_E)) =
\gamma(X,J_\ell(\pi_F),\mathcal{V}_F',
\psi_F)^l.
$$
Finally, the identity (\ref{univ_gamma_factors})
gives 
$$
(r_\ell\otimes \theta_\ell')(\gamma(X,\pi_E,\mathcal{V}',\psi_E)) =
\gamma(X^l,J_\ell(\pi_F)^{(l)}, \mathcal{V}_F'^{(l)},\psi_F^l).
$$
Thus, the space $\mathcal{N}(\psi_F)$ is stable under the action of
$G_n(F)$ and the map $\Phi_n$ is surjective. Now, using
\cite[Proposition 6.3]{dhar2022tate}, we get that there is a unique
generic subquotient of
$\widehat{H}^0(\mathbb{W}(\pi_E,\psi_{E,\overline{\mathcal{K}}})$,
and this is necessarily equal to $J_\ell(\pi_F)^{(l)}$.  This
completes the proof.
\end{proof}

\bibliography{fam} 

\providecommand{\bysame}{\leavevmode\hbox to3em{\hrulefill}\thinspace}
\providecommand{\MR}{\relax\ifhmode\unskip\space\fi MR }
\providecommand{\MRhref}[2]{%
  \href{http://www.ams.org/mathscinet-getitem?mr=#1}{#2}
}
\providecommand{\href}[2]{#2}
\begin{thebibliography}{Mos16b}

\bibitem[AC89]{Arthur_Clozel_BC}
James Arthur and Laurent Clozel, \emph{Simple algebras, base change, and the
  advanced theory of the trace formula}, Annals of Mathematics Studies, vol.
  120, Princeton University Press, Princeton, NJ, 1989. \MR{1007299}

\bibitem[DN22]{dhar2022tate}
Sabyasachi Dhar and Santosh Nadimpalli, \emph{Tate cohomology of whittaker
  lattices and base change of cuspidal representations of ${\rm gl}_n$}, arXiv
  preprint arXiv:2204.02131 (2022).

\bibitem[EH14]{Emerton_Helm_LLC_families}
Matthew Emerton and David Helm, \emph{The local {L}anglands correspondence for
  {${\rm GL}_n$} in families}, Ann. Sci. \'{E}c. Norm. Sup\'{e}r. (4)
  \textbf{47} (2014), no.~4, 655--722. \MR{3250061}

\bibitem[Fen23]{feng2023modular}
Tony Feng, \emph{Modular functoriality in the local langlands correspondence},
  arXiv preprint arXiv:2312.12542 (2023).

\bibitem[Fen24]{Feng_2024}
\bysame, \emph{Smith theory and cyclic base change functoriality}, Forum of
  Mathematics, Pi \textbf{12} (2024), e1.

\bibitem[Hel16a]{Helm_Bernstein_center}
David Helm, \emph{The {B}ernstein center of the category of smooth {$W(k)[{\rm
  GL}_n(F)]$}-modules}, Forum Math. Sigma \textbf{4} (2016), Paper No. e11, 98.
  \MR{3508741}

\bibitem[Hel16b]{Whittaker_model_Helm}
\bysame, \emph{Whittaker models and the integral {B}ernstein center for {${\rm
  GL}_n$}}, Duke Math. J. \textbf{165} (2016), no.~9, 1597--1628. \MR{3513570}

\bibitem[Hel20]{Curtis_Helm}
\bysame, \emph{Curtis homomorphisms and the integral {B}ernstein center for
  {${\rm GL}_n$}}, Algebra Number Theory \textbf{14} (2020), no.~10,
  2607--2645. \MR{4190413}

\bibitem[Hen93]{Henniart_conv_thm}
Guy Henniart, \emph{Caract\'{e}risation de la correspondance de {L}anglands
  locale par les facteurs {$\epsilon$} de paires}, Invent. Math. \textbf{113}
  (1993), no.~2, 339--350. \MR{1228128}

\bibitem[HM18]{Conv_thm_Helm_Moss}
David Helm and Gilbert Moss, \emph{Converse theorems and the local {L}anglands
  correspondence in families}, Invent. Math. \textbf{214} (2018), no.~2,
  999--1022. \MR{3867634}

\bibitem[JPSS79]{Aut_forms_GL3(I)}
Herv\'{e} Jacquet, Ilja~Iosifovitch Piatetski-Shapiro, and Joseph Shalika,
  \emph{Automorphic forms on {${\rm GL}(3)$}. {I}}, Ann. of Math. (2)
  \textbf{109} (1979), no.~1, 169--212. \MR{519356}

\bibitem[LM20]{Converse_thm_Moss_Liu}
Baiying Liu and Gilbert Moss, \emph{On the local converse theorem and the
  descent theorem in families}, Math. Z. \textbf{295} (2020), no.~1-2,
  463--483. \MR{4100015}

\bibitem[MM22]{Moss_Matringe_Kirillov_families}
Nadir Matringe and Gilbert Moss, \emph{The {K}irillov model in families},
  Monatsh. Math. \textbf{198} (2022), no.~2, 393--410. \MR{4421915}

\bibitem[Mos16a]{Gamma_factors_pairs_Moss}
Gilbert Moss, \emph{Gamma factors of pairs and a local converse theorem in
  families}, Int. Math. Res. Not. IMRN (2016), no.~16, 4903--4936. \MR{3556429}

\bibitem[Mos16b]{local_constant_Moss}
\bysame, \emph{Interpolating local constants in families}, Math. Res. Lett.
  \textbf{23} (2016), no.~6, 1789--1817. \MR{3621107}

\bibitem[Mos21]{Moss_nilpotent_gammafactors}
\bysame, \emph{Characterizing the mod-{$\ell$} local {L}anglands correspondence
  by nilpotent gamma factors}, Nagoya Math. J. \textbf{244} (2021), 119--135.
  \MR{4335904}

\bibitem[Ron16]{Ronchetti_local_bc}
Niccol\`o Ronchetti, \emph{Local base change via {T}ate cohomology}, Represent.
  Theory \textbf{20} (2016), 263--294. \MR{3551160}

\bibitem[TV16]{Venkatesh_Smith_theory}
David Treumann and Akshay Venkatesh, \emph{Functoriality, {S}mith theory, and
  the {B}rauer homomorphism}, Ann. of Math. (2) \textbf{183} (2016), no.~1,
  177--228. \MR{3432583}

\bibitem[Vig01]{Vigneras_semisimple_mod-l}
Marie-France Vign\'{e}ras, \emph{Correspondance de {L}anglands semi-simple pour
  {${\rm GL}(n,F)$} modulo {$l\not= p$}}, Invent. Math. \textbf{144} (2001),
  no.~1, 177--223. \MR{1821157}

\bibitem[Vig04]{Vigneras_integral_structure}
\bysame, \emph{On highest {W}hittaker models and integral structures},
  Contributions to automorphic forms, geometry, and number theory, Johns
  Hopkins Univ. Press, Baltimore, MD, 2004, pp.~773--801. \MR{2058628}

\end{thebibliography}
\bibliographystyle{amsalpha}

\noindent
Santosh Nadimpalli, \\
\texttt{nvrnsantosh@gmail.com}, \texttt{nsantosh@iitk.ac.in}.\\
Sabyasachi Dhar,\\
\texttt{mathsabya93@gmail.com},
\texttt{sabya@iitk.ac.in}\\
Department of Mathematics and Statistics,
Indian Institute of Technology Kanpur,
U.P. 208016, India.
\end{document}